\documentclass[11pt]{article}

\usepackage{amsmath,amssymb,amsthm}
\usepackage{tikz,hyperref}
\usetikzlibrary{arrows.meta}

\newtheorem{theorem}{Theorem}[section]

\newtheorem{corollary}[theorem]{Corollary}
\newtheorem{lemma}[theorem]{Lemma}
\newtheorem{question}[theorem]{Question}
\newtheorem{problem}[theorem]{Problem}

\theoremstyle{definition}

\newtheorem{example}[theorem]{Example}

\title{Wei duality, Fomin-Greene duality and demimatroids}
\author{Thomas Britz}
\date{}

\begin{document}

\maketitle

\begin{abstract}
The Fomin-Greene Duality Theorem for finite posets and Wei's Duality Theorem for demimatroids 
each relate sets of extremal invariants via seemingly similar dualities.
This paper describes precisely how these two dualities relate. 
It is shown that the chain and antichain numbers each satisfy 
a Wei-type duality with antichain- and chain-deletion numbers, respectively, 
and that, for each poset with at least two elements, 
its chain and antichain demimatroids are not related by any composition of standard demimatroid duality operations.
Their upper Wei-number numbers do however determine each other via Fomin-Greene duality.  
The two dualities thereby coincide for Wei numbers but not for the underlying rank functions.
\end{abstract}



\section{Introduction}
\label{sec:intro}

Wei's Duality Theorem~\cite{BrJoMaSh12,wei91} 
and the Fomin-Greene Duality Theorem~\cite{fomin78,greene76a} for finite posets
are celebrated theorems that in some ways seem strikingly similar, 
both relating two sets of extremal invariants by a dual operation. 
Alberto, Mart\'{i}nez-Bernal and Valencia-Bucio~\cite{AlMaVa26} recently showed that 
the chain and antichain numbers in the
Fomin-Greene Duality Theorem are upper Wei numbers of the antichain and chain demimatroids.
This raises a natural question:
how exactly do the two dualities underlying these theorems relate?

The present paper answers that question.
It shows that the chain and antichain numbers each satisfy a distinct Wei-type duality theorem.
It also shows that, in general, the chain and antichain demimatroids are not related by 
standard demimatroid involutions.
Nevertheless, the two duality theorems meet exactly for certain Wei numbers, 
and the paper describes this connection precisely.

The contents of the paper are as follows.
Section~\ref{sec:definitions} introduces the two theorems and the notation and terminology used to describe them.
Section~\ref{sec:question} presents the connection established in~\cite{AlMaVa26} 
and derives from it two Wei-type duality theorems for the chain and antichain numbers.
A question is posed, of whether the Fomin-Greene relation between these numbers is achievable by 
standard demimatroid dualities, and this question is answered in the negative.
The exact relation between the numbers is then described.
Section~\ref{sec:discussion} discusses a further involution operation between the chain and antichain complexes,
and an open problem is posed, regarding the possibility of a broader common context for the two duality theorems.

\section{The two duality theorems}
\label{sec:definitions}

Let $P = (E,\preceq)$ be a finite poset. 
A  {\em     chain} $C$ in $P$ is a subset of~$E$ for which either  $a\preceq b$  or $b\preceq a$ for all $a,b\in C$.
An {\em antichain} $A$ in $P$ is a subset of~$E$ for which neither $a\preceq b$ nor $b\preceq a$ for all $a,b\in A$.
For $k=0,1,\ldots$, let $c_k$ (resp., $a_k$) denote 
the maximal cardinality of a union of $k$ chains (resp., antichains) in~$P$, 
and let $\lambda_k=c_k-c_{k-1}$ and $\tilde{\lambda}_k=a_k-a_{k-1}$ for all $k\geq 1$.
The Fomin-Greene Duality Theorem for posets was discovered by Greene~\cite{greene76a} in 1976 and, 
independently, by Fomin~\cite{fomin78} in 1978; see also~\cite{BrFo01}.
It provides a duality that relates the chain numbers $c_i$ to the antichain number $a_j$, 
and thereby elegantly and powerfully generalises 
Dilworth's famous Theorem~\cite{dilworth50} as well as the dual of that theorem due to Mirsky~\cite{mirsky71}.


\begin{theorem}[The Fomin-Greene Duality Theorem~\cite{fomin78,greene76a}]
\label{thm:dual}$ $\\
For any finite poset $P$,
the sequences
$\lambda(P) = (\lambda_1, \lambda_2, \ldots)$ and
$\tilde{\lambda}(P) = (\tilde{\lambda}_1, \tilde{\lambda}_2, \ldots)$
are weakly decreasing, and form conjugate partitions of the number~$n=|P|$.
\end{theorem}

\begin{example}
\label{exa:Duality_Theorem_for_posets}
In the poset $P$ given by the Hasse diagram below, 
$c_1=4$, $c_2=6$, $a_1=2$, $a_2=4$, $a_3=5$ and $a_4=6$, 
so $\lambda(P) = (4,2)$ and $\tilde{\lambda}(P)=(2,2,1,1)$.
As asserted by Theorem~\ref{thm:dual}, $\tilde{\lambda}(P)=\lambda(P)^T$.
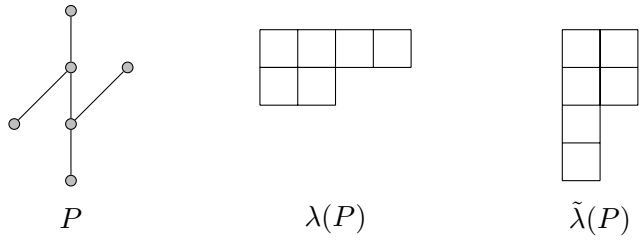
\[
\begin{tikzpicture}[scale=0.5]
  \begin{scope}[shift={(0,0)}]
	\tikzstyle{vertex}=[circle,fill=lightgray,draw=black,inner sep = 0.5mm]
    \foreach \i/\x/\y in {a/0/0, b/-1/1, c/0/1, d/0/2, e/1/2, f/0/3}{\node[vertex] (\i) at (1.5*\x,1.5*\y) {};}
    \draw (a) -- (c) -- (e)  
          (b) -- (d) -- (f) 
          (c) -- (d);
    \draw (0,-1) node {$P$};
  \end{scope}
  \begin{scope}[shift={(5,0)}]
    \draw (0, 3) grid (4,4)
          (0, 2) grid (2,3);
    \draw (2,-1) node {$\lambda(P)$};
  \end{scope}
  \begin{scope}[shift={(12,0)}]
    \draw (1, 0) grid (2,4)
          (2, 2) grid (3,4);
    \draw (2,-1) node {$\tilde{\lambda}(P)$};
  \end{scope}
\end{tikzpicture}
\]
\end{example}

For any $[n,k]$ linear code $C$ over some field~$\mathbb{F}$,
the $i$th (resp.\ $j$th) {\em generalised (Hamming) weight} of $C$ (resp.\ $C^\perp$) is,
for $i = 0,\ldots,k$ and $j = 0,\ldots,n-k$,
\[
  \begin{array}{r@{\;\,}l@{\;}l}
    d_i       & := d_i(C)      &= \min\bigl\{\textrm{wt}(D)\::\: \textrm{$D$ is a linear $[n,i]$ subcode of $C$}\}\,;\\[1mm]
    d^\perp_j & := d_j(C^\perp)&= \min\bigl\{\textrm{wt}(D)\::\: \textrm{$D$ is a linear $[n,j]$ subcode of $C^\perp$}\}\,.
  \end{array}
\]
In 1991, Wei~\cite{wei91} presented a remarkably elegant and useful dual relationship 
between the numbers $d_i$ and $d^\perp_j$, 
as expressed by the following theorem, 
in which $[n]:= \{1,\ldots,n\}$.

\begin{theorem}[Wei's Duality Theorem~\cite{wei91}]
\label{TBPJC:thm:Wei}
For each $[n,k]$ linear code $C$ over a field~$\mathbb{F}$,
set $U_C := \{d_1,\ldots,d_k\}$ and $V_C := \{n+1-d^\perp_{n-k},\ldots,n+1-d^\perp_1\}$.
Then
\[
  U_C\cup V_C = [n] \quad\textrm{and}\quad
  U_C\cap V_C = \emptyset\,.
\]
\end{theorem}

\begin{example}
\label{exa:Wei-linear-code}
The $[5,3]$ linear code $C$ over $\mathbb{Z}_2$ given the generator matrix 
\[
  \begin{pmatrix}
    1 & 0 & 1 & 0 & 0\\
    0 & 1 & 1 & 0 & 0\\
    0 & 0 & 0 & 1 & 1
  \end{pmatrix}
\]
has generalised weights $d_1 = 2$, $d_2 = 3$ and $d_3 = 5$.
Similarly, the dual code $C^\perp$ has generalised weights $d^\perp_1 = 2$ and $d^\perp_2 = 5$.
Therefore,
\begin{align*}
  U_C &:= \{d_1,d_2,d_3\} = \{2,3,5\}\\
  V_C &:= \{5+1-d^\perp_2,\ldots,5+1-d^\perp_1\} = \{6-5,6-2\} = \{1,4\}
\end{align*}
and so $U_C\cup V_C = [5]$ 
and    $U_C\cap V_C = \emptyset$,
as asserted by Theorem~\ref{TBPJC:thm:Wei}.
\end{example}

Wei's Duality Theorem has been generalised in many ways; see for instance
\cite{ashikhmin98, BrJoMaSh12, BrMaSh20, ducoat2015, GhJo20, HoSh01, martinex-penas19, MaMa18, 
MoFi10, PaPrRa23, ravagnani16,tang24, XuKaHa22}.
The demimatroid generalisations in~\cite{BrJoMaSh12} are particularly important 
since demimatroids, defined below, provide a natural and fundamental combinatorial framework for 
understanding and further generalising Wei's Duality Theorem.


A {\em demimatroid} is a pair $D=(E,r)$ where $E$ is a finite set and $r:2^E\to\mathbb{Z}$ is 
a function satisfying the following two conditions:
\begin{itemize}\setlength{\itemsep}{0mm}
  \item[{(R1)}] $r(\emptyset) = 0$\,;
  \item[{(R2)}] for each $X\subseteq E$ and $x\in E$, \, 
                $r(X)\leq r(X\cup\{x\})\leq r(X)+1$\,.
\end{itemize}
The {\em rank} of $D$ is $r(E)$.

\begin{example}
\label{exa:vector-demimatroid}
Let $C$ be a $[n,k]$ linear code over a field $\mathbb{F}$.
The {\em vector demimatroid} of $C$ is the demimatroid $D_C = (E,r)$
where $E = [n]$ and where $r(A)$ is the dimension of the punctured code $C\setminus (E-A)$, 
i.e., the code obtained by deleting from each codeword in $C$ any entry not indexed by an element of~$A$.
Note that $D_C$ is also the usual {\em vector matroid} of $C$; see~\cite{oxley11}.
\end{example}

\begin{example}
\label{exa:chain-antichain-demimatroid}
Let $P=(E,\preceq)$ be a finite poset
and let $\mathcal{C}:=\mathcal{C}(P)$ and $\mathcal{A}:=\mathcal{A}(P)$ 
be the families of chains and antichains of $P$, respectively.
Define $D_\mathcal{C} = (E,r_\mathcal{C})$ and $D_\mathcal{A} = (E,r_\mathcal{A})$ 
where, for each $X\subseteq E$,
\begin{align*}
  r_\mathcal{C}(X) &= \max\{\,|C| \,:\, C\subseteq X\,,\; C\in\mathcal{C}\}\,;\\
  r_\mathcal{A}(X) &= \max\{\,|A| \,:\, A\subseteq X\,,\; A\in\mathcal{A}\}\,.    
\end{align*}
Then $D_\mathcal{C}$ and $D_\mathcal{A}$ are demimatroids, 
called the {\em chain demimatroid} and {\em antichain demimatroid} of $P$, respectively; 
see~\cite{AlMaVa26}.
Note that $r_\mathcal{C}(E) = c_1(P)$ and $r_\mathcal{A}(E) = a_1(P)$,
and that demi-matroids can be defined as above for simplicial complexes more generally.
\end{example}

Below are three duality operations on demimatroids, 
mapping $D=(E,r)$ as follows:
\[
  \begin{array}{l@{\;\;}l@{\;}lll@{\;}l}
    \text{The {\em dual}       of } D: & D^*          &= (E,r^*)            &\text{where} & r^*(A)          &= |A|+r(E\setminus A)-r(E)\,;\\
    \text{The {\em supplement} of } D: & \overline{D} &= (E,\overline{r}\;) &\text{where} & \overline{r}\,(A) &= r(E)-r(E\setminus A)\,;\\
    \text{The {\em nullity}    of } D: & D^\circ      &= (E,r^\circ)        &\text{where} & r^\circ(A)      &= |A|-r(A)\,.
  \end{array}
\]
As shown in~\cite{BrJoMaSh12,MaVaVi23}, 
these three operations each yield a new demimatroid, and each operation is an involution. 
Also the three operations commute, and any two compose to the third, 
forming, together with the identity, the Klein group of operations on rank functions.

Now define 
\begin{align*}
  \overline{\sigma}_i &= \min\{|X|\;:\; X\subseteq E\,,\;\overline{r}\,(X) = i\}\,;\\
  \overline{\tau}_j   &= \min\{|X|\;:\; X\subseteq E\,,\;\overline{r}^*(X) = j\}\,.
\end{align*} 
In~\cite{BrJoMaSh12}, it was shown that the generalised weights of a linear code~$C$ 
are determined by the vector demimatroid $D_C$; 
in particular, $d_i=\overline{\sigma}_i$ and $d^\perp_j=\overline{\tau}_j$.
Wei's Duality Theorem is therefore naturally generalised with respect to demimatroids as follows.

\begin{theorem}[Wei duality for demimatroids~\cite{BrJoMaSh12}]
\label{thm:Wei-demi-matroids}$ $\\
For each demimatroid $D = (E,r)$ of rank $k = r(E)$ and size $n=|E|$,
define sets
$U_D := \{\overline{\sigma}_1,\ldots,\overline{\sigma}_k\}$ and 
$V_D := \{n+1-\overline{\tau}_{n-k},\ldots,n+1-\overline{\tau}_1\}$.
Then 
\[
  U_D\cup V_D = [n] \quad\textrm{and}\quad
  U_D\cap V_D = \emptyset\,.
\]
\end{theorem}

\section{How the two dualities relate}
\label{sec:question}

Following~\cite{AlMaVa26}, 
define the $i$th \emph{upper Wei number} of a demimatroid $D=(E,r)$ by
\[
  \sigma^i(D):=\max\{\,|X|\::\:X\subseteq E,\ r(X)=i\,\}\,.
\]
By definition, these numbers relate to the numbers $\overline{\sigma}_{i}(D)$ as follows, 
where $n=|E|$ and $k=r(E)$.
\begin{lemma}
\label{lem:upper-lower}
For each $i=0,1,\ldots,k$,\, $\sigma^i(D) = n - \overline{\sigma}_{k-i}(D) = n - \overline{\tau}_{k-i}(D^*)$.
\end{lemma}

Alberto, Mart\'{i}nez-Bernal and Valencia-Bucio~\cite[Propositions~5.4 and~5.7]{AlMaVa26}
recently showed a direct connection between the numbers occurring in the two duality theorems.

\begin{theorem}
\label{thm:upper-wei-posets}
Let $P$ be a finite poset of height $h=c_1(P)$ and width $w=a_1(P)$.
For all $i=0,1,\ldots,w$ and $j=0,1,\ldots,h$,
\begin{align*}
  c_i(P) &= \sigma^i(D_\mathcal{A}(P))\,;\\
  a_j(P) &= \sigma^j(D_\mathcal{C}(P))\,.
\end{align*}
\end{theorem}

That is, 
the chain and antichain numbers in the Fomin-Greene Duality Theorem 
are upper Wei numbers of the antichain and chain demimatroids, respectively.
This provides a precise connection between the two theories.
It also provides two explicit Wei-type duality theorems for the chain and antichain numbers, as follows.

Let $P=(E,\preceq)$ be a finite poset.
Define the {\em     chain-deletion numbers} $\delta^{\mathcal C}_i(P)$ 
   and the {\em antichain-deletion numbers} $\delta^{\mathcal A}_j(P)$ of $P$ respectively as follows:
\begin{align*}
  \delta^{\mathcal C}_i(P) &:=\min\{\,|X|\::\:X\subseteq E\,,\ r_\mathcal{C}^\circ(X) = i\,\}\,;\\[.5mm]
  \delta^{\mathcal A}_j(P) &:=\min\{\,|X|\::\:X\subseteq E\,,\ r_\mathcal{A}^\circ(X) = j\,\}\,.
\end{align*}
The number $r_\mathcal{C}^\circ(X) = |X|-r_\mathcal{C}(X)$
   (resp., $r_\mathcal{A}^\circ(X) = |X|-r_\mathcal{A}(X)$) is 
the minimum number of elements that must be deleted from $X$ to leave a chain (resp., an antichain).
Together with the chain and antichain numbers, 
the deletion numbers $\delta^{\mathcal A}_t(P)$ and $\delta^{\mathcal C}_s(P)$ 
satisfy Wei-type dualities that can be purely expressed in terms of the poset~$P$.

\begin{theorem}
\label{thm:poset-wei-dualities}
Let $P$ be a finite poset of size $n$, height $h$ and width $w$,
and define 
\[
  \begin{array}{l@{\;}l@{\qquad}l@{\;}l}
    U_C &:= \{ c_0(P)+1,\ldots,c_{w-1}(P)+1\} &  V_A &:= \{ \delta_1^\mathcal{A}(P),\ldots,\delta_{n-w}^\mathcal{A}(P)\}\\[0.5mm]
    U_A &:= \{ a_0(P)+1,\ldots,a_{h-1}(P)+1\} &  V_C &:= \{ \delta_1^\mathcal{C}(P),\ldots,\delta_{n-h}^\mathcal{C}(P)\}\,.
  \end{array}
\]  
Then $U_C \cup V_A = U_A \cup V_C = [n]$ and $U_C\cap V_A = U_A\cap V_C = \emptyset$\,.
\end{theorem}

\begin{example}
\label{exa:theorem-poset-Wei}
For the poset $P$ in Example~\ref{exa:Duality_Theorem_for_posets},
\[
  \begin{array}{l@{\;}l@{\qquad}l@{\;}l}
    U_C &= \{ 1,5\}      &  V_A &= \{ 2,3,4,6\}\\
    U_A &= \{ 1,3,5,6\}  &  V_C &= \{ 2,4\}\,.
  \end{array}
\]  
Then $U_C \cup V_A = U_A \cup V_C = [6]$ and $U_C\cap V_A = U_A\cap V_C = \emptyset$, 
as asserted in Theorem~\ref{thm:poset-wei-dualities}.
\end{example}

\begin{proof}[Proof of Theorem~\ref{thm:poset-wei-dualities}]
Let $D_\mathcal{A} := D_\mathcal{A}(P) = (E,r_\mathcal{A})$
and note that $r_\mathcal{A}^*(E) = n-w$.
By Theorem~\ref{thm:upper-wei-posets} and Lemma~\ref{lem:upper-lower}, 
$c_i(P) + 1 = \sigma^i(D_\mathcal{A}) + 1 = n + 1 - \overline{\sigma}_{w-i}(D_\mathcal{A})$.
Also,
\begin{align*}
     \delta^{\mathcal A}_j(P) 
  &=     \min\{\,|X|\::\:X\subseteq E\,,\ |X| - r_\mathcal{A}(X) = j\,\}\\
  &= n - \max\{\,|X|\::\:X\subseteq E\,,\ r_\mathcal{A}^*(X) = n-w-j\,\}\\
  &= n - \sigma^{n-w-j}(D_\mathcal{A}^*)\,,
\end{align*}
so, by Lemma~\ref{lem:upper-lower}, 
$\delta^{\mathcal A}_j(P) = \overline{\tau}_{j}(D_\mathcal{A})$.
Similarly, 
$a_j(P) + 1 = n + 1 - \overline{\sigma}_{h-j}(D_\mathcal{C})$ and 
$\delta^{\mathcal C}_j(P) = \overline{\tau}_{j}(D_\mathcal{C})$.
The proof therefore follows from Theorem~\ref{thm:Wei-demi-matroids}.
\end{proof}

By Theorem~\ref{thm:poset-wei-dualities}, 
the chain and antichain numbers each separately satisfy a Wei-type duality theorem.
This follows from Theorem~5.2 and Propositions~5.4 and~5.7 of~\cite{AlMaVa26},
although it does not seem to have been stated there explicitly.
However, the chain and antichain numbers are not Wei duals of each other: 
the dualities pairs chain and anti numbers with antichain-deletion and chain-deletion numbers, respectively.
The three relations may therefore be summarised as follows:
\begin{center}
\begin{tikzpicture}[>=Stealth,thick]
  \draw (0,3  ) node {$\big(c_i(P)\big)$};
  \draw (5,3  ) node {$\big(a_j(P)\big)$};
  \draw (0,1  ) node {$\big(\delta^{\mathcal A}_s(P)\big)$};
  \draw (5,1  ) node {$\big(\delta^{\mathcal C}_t(P)\big)$};
  \draw (2.5,3.25) node {\footnotesize Fomin-Greene};
  \draw ( .33,2) node[rotate=90] {\footnotesize Wei};
  \draw (4.67,2) node[rotate=90] {\footnotesize Wei};
  \draw[<->]  (1,3) -- (4,3); 
  \draw[<->] (0,1.4) -- (0,2.6); 
  \draw[<->] (5,1.4) -- (5,2.6); 
\end{tikzpicture}
\end{center}
Since $c_i(P) = \sigma^i(D_\mathcal{A}(P))$ and $a_j(P) = \sigma^j(D_\mathcal{C}(P))$ in 
Theorem~\ref{thm:upper-wei-posets}, 
the natural question is whether the horizontal relation can be explained by 
standard demimatroid operations on $D_\mathcal{C}(P)$ and $D_\mathcal{A}(P)$.

\begin{question}
\label{que:demimatroid-duality}
For a finite poset $P$, can the antichain demimatroid $D_\mathcal{A}(P)$ be obtained from the chain demimatroid
$D_\mathcal{C}(P)$ by a composition of the dual, supplement and nullity operations?
\end{question}

Since the identity, dual, supplement and nullity operations form a group under composition, 
Theorem~\ref{thm:no-standard-duality} below answers Question~\ref{que:demimatroid-duality} in the negative.
The one-element poset is the only nonempty exception: 
its chain and antichain demimatroids coincide.
Dualising the poset $P$ does not change this conclusion, 
since chains and antichains remain unchanged under this operation.
Nor does restricting to a proper ideal or filter change the conclusion
since this would change the ground set $E$ and therefore not yield~$D_\mathcal{A}(P)$.

\begin{theorem}
\label{thm:no-standard-duality}
Let $P = (E,\preceq)$ be a finite poset with at least two elements.
Then
\[
  D_\mathcal{A}(P)\notin
  \big\{D_\mathcal{C}(P),D_\mathcal{C}(P)^*,\overline{D_\mathcal{C}(P)},D_\mathcal{C}(P)^\circ\big\}\,.
\]
\end{theorem}

\begin{proof}
Write $D_\mathcal{C}(P) = (E,r_\mathcal{C})$ and
      $D_\mathcal{A}(P) = (E,r_\mathcal{A})$, 
and let $n=|E|$, $h = r_\mathcal{C}(E)$ and $w = r_\mathcal{A}(E)$
be the size, height and width of $P$, respectively.

Let $x$ and $y$ be distinct elements of $P$. 
Then either $\{x,y\}$ is a      chain, in which case $r_\mathcal{C}(\{x,y\})=2$ and $r_\mathcal{A}(\{x,y\})=1$, 
         or $\{x,y\}$ is an antichain, in which case $r_\mathcal{C}(\{x,y\})=1$ and $r_\mathcal{A}(\{x,y\})=2$.
In either case, $r_\mathcal{C}(\{x,y\})\neq r_\mathcal{A}(\{x,y\})$, 
so $D_\mathcal{A}(P)\neq D_\mathcal{C}(P)$.
Also,
$       r_\mathcal{C}^\circ(\{x\})
  = 1 - r_\mathcal{C}(\{x\}) 
  = 0 
 \neq 1 
  = r_\mathcal{A}(\{x\})
$,
so $D_\mathcal{A}(P)\neq D_\mathcal{C}(P)^\circ$.

Assume that $D_\mathcal{A}(P)=\overline{D_\mathcal{C}(P)}$.
Then $w = r_\mathcal{A}(E) = \overline{r_\mathcal{C}}(E) = r_\mathcal{C}(E)-r_\mathcal{C}(\emptyset) = h$
and, for each $x\in E$,
$
  1 = r_\mathcal{A}(\{x\})
    = \overline{r_\mathcal{C}}(\{x\})
    = r_\mathcal{C}(E)-r_\mathcal{C}(E\setminus\{x\})
    = h-r_\mathcal{C}(E\setminus\{x\})
$,
so every chain of cardinality $h$ contains~$x$.
Therefore, every element of $E$ belongs to every chain of cardinality $h$, so $P$ is a chain.
Then $h=n$ and $w=1$, so $1=w=h=n\geq 2$, a contradiction.
Hence, $D_\mathcal{A}(P)\neq\overline{D_\mathcal{C}(P)}$.

Finally, assume that $D_\mathcal{A}(P)=D_\mathcal{C}(P)^*$.
Then $w = r_\mathcal{A}(E) = r_\mathcal{C}^*(E) = n - h$.
Let $C$ be a chain of size $h$ and set $A := E\setminus C$.
Then 
$          r_\mathcal{C}(A)
   = r_\mathcal{C}(A) + |C| - h
   = r_\mathcal{C}^*(C)
   = r_\mathcal{A}(C)
   = 1
$,
so $A$ is an antichain.
Furthermore, $|A|=n-h=w$, so $A$ is a maximal antichain.

For each $c\in C$, 
$r_\mathcal{C}(E\setminus\{c\}) = r_\mathcal{C}(E) - |\{c\}| + r_\mathcal{A}(\{c\}) = h$, 
so $E\setminus\{c\}$ contains a chain $C_c$ of cardinality $h$.
Since $A$ is an antichain, 
$C_c = (C\setminus\{c\})\cup\{c'\}$
for some $c'\in A$.
Let the function $f:C\to A$ be given by $f(c) = c'$, 
and suppose that $f(c)=f(d)$ for distinct $c,d\in C$.
Then the chains $C_c$ and $C_d$ each contain $f(c)$, 
so $C\cup\{f(c)\}$ is a chain of cardinality $h+1$, a contradiction.
Therefore, $f$ is injective, and so $h \leq w$.

Since $D_\mathcal{C}(P)=D_\mathcal{A}(P)^*$, 
swapping chains and antichains in the argument above yields an injection $g:A\to C$, 
and so $w\leq h$.
Thus, $h=w$, and $f$ and $g$ are bijections.

For each $c\in C$, the element $f(c)$ is comparable with every element $d\in C\setminus\{c\}$; 
that is, either $f(c)\preceq d$ or $d\preceq f(c)$; 
and $f(c)$ is incomparable with~$c$.
Since $f$ is bijective, 
each $c\in C$ is incomparable with exactly one element of $A$. 
On the other hand,
$(A\setminus\{a\})\cup\{g(a)\}$
is an antichain for each $a\in A$, 
so $g(a)$ is incomparable with all $w-1$ elements of $A\setminus\{a\}$. 
Thus, $w-1\leq 1$ and so $h=w\leq2$.
If $h=1$, then the two elements of $E=C\cup A$ are incomparable, contradicting $w=1$.
If $h=2$, then write $C=\{c_1,c_2\}$ and $f(c_i)=a_i$.
Then $\{c_1,a_1\}$ and $\{c_2,a_2\}$ are antichains,
so $r_\mathcal{A}(\{c_1,a_1\})=2$,
whereas
$r_\mathcal{C}^*(\{c_1,a_1\}) = 2+r_\mathcal{C}(\{c_2,a_2\})-2=1$,
a contradiction.
Hence, $D_\mathcal{A}(P)\neq D_\mathcal{C}(P)^*$.
\end{proof}

Theorem~\ref{thm:poset-wei-dualities} shows that the chain and antichain numbers each satisfy a Wei-type duality, 
and 
Theorem~\ref{thm:no-standard-duality} shows that their Fomin-Greene relation is not induced by 
any standard demimatroid duality operation.
It remains to determine the precise relation between the two types of upper Wei numbers.

Theorem~\ref{thm:upper-wei-posets} and Lemma~\ref{lem:upper-lower} allow 
the Fomin-Greene Duality Theorem to be expressed purely in terms of demimatroid Wei numbers.

\begin{theorem}
\label{thm:FG-demi-form}
Let $P=(E,\preceq)$ be a finite poset of height $h$ and width $w$.
Then 
\begin{align*}
  &\big(\overline{\sigma}_w(D_\mathcal{A})-\overline{\sigma}_{w-1}(D_\mathcal{A}), \ldots,
        \overline{\sigma}_1(D_\mathcal{A})-\overline{\sigma}_0(D_\mathcal{A})\big)\,,\\
  &\big(\overline{\sigma}_h(D_\mathcal{C})-\overline{\sigma}_{h-1}(D_\mathcal{C}), \ldots,
        \overline{\sigma}_1(D_\mathcal{C})-\overline{\sigma}_0(D_\mathcal{C})\big)
\end{align*}
are conjugate partitions of $n=|P|$, where
$D_\mathcal{C}=D_\mathcal{C}(P)$ and
$D_\mathcal{A}=D_\mathcal{A}(P)$.
\end{theorem}

Thus, each of the two sets of upper Wei numbers determines the other.
Theorem~\ref{thm:dual} and Theorem~\ref{thm:upper-wei-posets} give the following explicit identities.
\begin{corollary}\label{cor:explicit}
For each $i=0,1,\ldots,h$ and $j=0,1,\ldots,w$, 
\begin{align*}
     \sigma^i(D_\mathcal{C})
  &= \sum_{j=1}^{w} \min\big\{i,\sigma^j(D_\mathcal{A})-\sigma^{j-1}(D_\mathcal{A})\big\}\,,\\
     \sigma^j(D_\mathcal{A})
  &= \sum_{i=1}^{h} \min\big\{j,\sigma^i(D_\mathcal{C})-\sigma^{i-1}(D_\mathcal{C})\big\}\,.
\end{align*}
\end{corollary}

\begin{example}
\label{exa:positive-connection}
For the poset $P$ in Example~\ref{exa:Duality_Theorem_for_posets},
\begin{align*}
    (\sigma^0(D_\mathcal{A}),\sigma^1(D_\mathcal{A}),\sigma^2(D_\mathcal{A}))
  = (c_0,c_1,c_2)
  &=(0,4,6)\\\quad\hbox{and}\quad
    (\sigma^0(D_\mathcal{C}),\ldots,\sigma^4(D_\mathcal{C})) 
  = (a_0,a_1,a_2,a_3,a_4)
  &=(0,2,4,5,6)\,,
\end{align*}
By Lemma~\ref{lem:upper-lower},
\begin{align*}
 (\overline{\sigma}_0(D_\mathcal{A}),
  \overline{\sigma}_1(D_\mathcal{A}),
  \overline{\sigma}_2(D_\mathcal{A}))
 &= (0,2,6)\\\text{and}\qquad\qquad
 (\overline{\sigma}_0(D_\mathcal{C}),\ldots,
  \overline{\sigma}_4(D_\mathcal{C}))
 &= (0,1,2,4,6)\,.
\end{align*}
The reversed difference sequences are therefore
\[
  (6-2,2-0)=(4,2)
  \quad\hbox{and}\quad
  (6-4,4-2,2-1,1-0)=(2,2,1,1),
\]
which are the conjugate partitions in Example~\ref{exa:Duality_Theorem_for_posets}.
\end{example}

Theorems~\ref{thm:poset-wei-dualities}, \ref{thm:no-standard-duality}
and~\ref{thm:FG-demi-form} provide the answer.
The chain and antichain numbers each satisfy a Wei-type duality,
paired respectively with the antichain- and chain-deletion numbers,
while the chain and antichain numbers determine each other by Fomin-Greene duality.
The nature of these dualities are distinct: 
Wei duality relates aspects of a demimatroid to those of its dual, 
whereas Fomin-Greene duality relates $D_\mathcal{C}(P)$ and $D_\mathcal{A}(P)$, 
which are not standard demimatroid duals.
The two theories meet exactly at these upper Wei-numbers,
not at the level of the underlying rank functions.

\section{A further involution and an open problem}
\label{sec:discussion}

There is also a precise relation for simplicial complexes.
For a simplicial complex $\Delta$ on~$E$, define its \emph{anticomplex}~\cite{AlMaVa26} by
\[
  \Delta^\perp
  :=\{\,\tau\subseteq E\::\: |\tau\cap\sigma|\leq 1 \; \text{for every}\; \sigma\in\Delta\,\}\,.
\]
Then $\mathcal{C}(P)^\perp=\mathcal{A}(P)$ and $\mathcal{A}(P)^\perp=\mathcal{C}(P)$, 
since a set meets every chain in at most one element exactly when it is an antichain.  
Anticomplex orthogonality is an involution that swaps the two complexes, 
but it is not one of the standard demimatroid operations 
and it is not an involution on arbitrary simplicial complexes: 
in general, only $\Delta\subseteq\Delta^{\perp\perp}$~\cite{AlMaVa26}.

There is then no immediate relation between the two dualities, 
and it appears that the Fomin-Greene Duality Theorem, or Theorem~\ref{thm:FG-demi-form}, 
cannot be generalised to hold for a demimatroid and some natural dual thereof.
However, a broader demimatroid generalisation might be possible, 
requiring paired rank data or a two-dimensional invariant as in 
the Fomin-Greene Ferrers diagram.
The pair $\big(D_\mathcal{C}(P),D_\mathcal{A}(P)\big)$ does not provide this: 
its horizontal relation is Fomin-Greene conjugation, 
whereas Wei duality sends each member to a different deletion-number.

\begin{problem}
\label{prob:paired-ranks}
Find a natural class of demimatroids pairs $(D_1,D_2)$ for which 
the Wei-number gap sequences are conjugate partitions, 
and determine whether this class admits an involution having 
both Fomin-Greene duality and Wei duality as special cases.
\end{problem}

\bibliographystyle{plain}

\noindent
{\sc School of Mathematics and Statistics, UNSW Sydney, Sydney, NSW 2052, Australia}\\
{\em Email address: \url{britz@unsw.edu.au}}

\end{document}